\documentclass[12pt]{amsart}
\usepackage[colorlinks=true,citecolor=blue,linkcolor=blue]{hyperref}
\usepackage{latexsym,amsmath,amssymb}
\usepackage{accents}
\usepackage{a4wide}
\usepackage{color}
\usepackage[colorinlistoftodos,prependcaption,textsize=tiny]{todonotes}
  
\title{On regularity theory for $n/p$-harmonic maps into manifolds}

\author{Francesca Da Lio}
\author{Armin Schikorra}

\address[Francesca Da Lio]{Department of Mathematics, ETH Z\"urich, R\"amistrasse 101, 8092 Z\"urich, Switzerland.}
\email{fdalio@math.ethz.ch}

\address[Armin Schikorra]{Department of Mathematics, 301 Thackeray Hall, University of Pittsburgh, PA 15260, USA}
\email{armin@pitt.edu}

plus 6pt minus 12pt
\def\eps{\varepsilon}

\def\N{{\mathbb N}}

\newtheorem{theorem}{Theorem}

\newtheorem{lemma}[theorem]{Lemma}

\newtheorem{proposition}[theorem]{Proposition}

\theoremstyle{definition}

\newcommand{\R}{\mathbb{R}}

\newcommand{\brac}[1]{\left (#1 \right )}

\newcommand{\barint}{
\rule[.036in]{.12in}{.009in}\kern-.16in \displaystyle\int }

\newcommand{\barcal}{\mbox{$ \rule[.036in]{.11in}{.007in}\kern-.128in\int $}}

\def\mvint_#1{\mathchoice
          {\mathop{\vrule width 6pt height 3 pt depth -2.5pt
                  \kern -8pt \intop}\nolimits_{\kern -3pt #1}}%
          {\mathop{\vrule width 5pt height 3 pt depth -2.6pt
                  \kern -6pt \intop}\nolimits_{#1}}%
          {\mathop{\vrule width 5pt height 3 pt depth -2.6pt
                  \kern -6pt \intop}\nolimits_{#1}}%
          {\mathop{\vrule width 5pt height 3 pt depth -2.6pt
                  \kern -6pt \intop}\nolimits_{#1}}}

\numberwithin{theorem}{section} \numberwithin{equation}{section}

\renewcommand{\div}{\operatorname{div}}

\newcommand{\lap}{\Delta }
\newcommand{\aleq}{\lesssim}

\newcommand{\aeq}{\simeq}
\newcommand{\Rz}{\mathcal{R}}
\newcommand{\laps}[1]{(-\lap)^{\frac{#1}{2}}}
\newcommand{\lapv}{(-\lap)^{\frac{1}{4}}}

\newcommand{\lapms}[1]{I^{#1}}

\begin{document}

\sloppy
\keywords{Fractional harmonic maps, nonlinear elliptic PDE's, regularity of solutions, commutator estimates.}
\subjclass[2010]{58E20, 35J20, 35B65, 35J60, 35S99}
\sloppy


\begin{abstract}
In this paper we continue the investigation started in the paper \cite{DaLio-Schikorra}  of the regularity of the so-called weak
$\frac{n}{p}$-harmonic maps in the critical case. These are  critical points of the following nonlocal energy
\[
{\mathcal{L}}_s(u)=\int_{\R^n}| ( {-\Delta})^{\frac{s}{2}} u(x)|^p dx\,,
\]
where $u\in \dot{H}^{s,p}(\R^n,\mathcal{N})$ and ${\mathcal{N}}\subset\R^N$ is a closed  $k$ dimensional smooth manifold  and $s=\frac{n}{p}$. 
We prove H\"older continuity for such critical points for $p \leq 2$. For $p > 2$ we obtain the same under an additional Lorentz-space assumption.
The regularity theory is in the two cases based on regularity results for nonlocal Schr\"odinger systems with an antisymmetric potential.
\end{abstract}

\maketitle
\tableofcontents

\section{Introduction}

Half-harmonic maps were first studied by Rivi\`{e}re and the first-named author \cite{DLR1,DLR2}. The $L^2$-regularity theory has been extended to higher dimension \cite{Millot-Sire-2015,Schikorra-n2,DLnD,Schikorra-eps}, and to $L^p$-energies \cite{DaLio-Schikorra,Schikorra-gagliardo,Schikorra-lpgrad}. Compactness and quantization issues have been addressed \cite{DLbubble,DaLio-Laurain-Riviere-2016}.

Here we extend our analysis of weak $n/p$-harmonic maps initiated in  \cite{DaLio-Schikorra} in the sphere case to general target manifolds. 

They are critical points of the energy 
\begin{equation}\label{Energy}
{\mathcal{L}}_s(u)=\int_{\R^n}| ( {-\Delta})^{\frac{s}{2}} u(x)|^p dx
\end{equation}
acting on maps $u \in \dot{H}^{s,p}(\R^n,\R^N)$ which pointwise map into a smooth, closed   (compact  and  without boundary)  $k$-dimensional manifold  $\mathcal{N} \subset \R^N$. This class of maps is commonly denoted by $\dot{H}^{s,p}(\R^n,\mathcal{N})$.  We will refer to Section~\ref{defnot} for the precise definition of such functional spaces.
The Euler-Lagrange equation for critical points $u$ can be formulated as follows  \begin{equation}\label{eq:eleq}
\Pi(u)( \laps{s} \brac{|\laps{s} u|^{p-2} \laps{s}u})=0  ~~\mbox{in ${\mathcal{D}}'(\R^n)$}\end{equation}
where $\Pi\colon {\mathcal{U}}_\delta\to  \mathcal{N}$ is the standard nearest point projection of a $\delta$-neighborhood ${\mathcal{U}}_\delta$ of $ \mathcal{N}$ onto   $\mathcal{N}$.
  Our first main result is the regularity theory for the case $p < 2$.
\begin{theorem}\label{th:pl2}
Assume that $u \in \dot{H}^{s,\frac{n}{s}}(\R^n,\mathcal{N})$, for $\frac{n}{s} \leq 2$, is a solution to \eqref{eq:eleq}.
Then $u$ is locally H\"older continuous.
\end{theorem}
The case $p>2$ presents additional difficulties.
  Here we show, that under the additional assumption that $\laps{s} u$ belongs to the smaller Lorentz space $L^{(p,2)} \subset L^p$, regularity theory follows. More precisely we have
\begin{theorem}\label{th:pg2}
Assume that $u \in \dot{H}^{s,\frac{n}{s}}(\R^n,\mathcal{N})$, for $\frac{n}{s} \geq 2$, is a solution to \eqref{eq:eleq}. If we additionally assume
\begin{equation}\label{eq:extraassumption}
 \|\laps{s} u\|_{(\frac{n}{s},2)} < \infty,
\end{equation}
then $u$ is locally H\"older continuous.
\end{theorem}
Let us stress that the extra assumption \eqref{eq:extraassumption} is not motivated by geometric arguments, but by pure analytic considerations, and we do not know if \eqref{eq:extraassumption} is a necessary assumption. Indeed this is related to a major open problem, the regularity theory of $n$-harmonic maps into manifolds and generalized $H$-systems, see \cite{Schikorra-Strzelecki}. Also in that case, regularity can only be proven under additional analytic assumptions that cannot be justified geometrically, see \cite{Kolasinksi-2010,Schikorra-2013}. However, these additional assumptions do not a priori rule out the possible singularities such as $\log \log 1/|x|$, so the geometric structure of the Euler-Lagrange equation plays an important role. 

Both theorems follow from a reduction to a system with antisymmetric structure, in the spirit of Rivi\`{e}re's seminal work \cite{Riviere-2007} which was adapted to nonlocal equations first by Rivi\`{e}re and the first-named author \cite{DLR2}, for related arguments see also \cite{DLnD,Mazowiecka-Schikorra}. Namely we have

\begin{proposition}\label{pr:reformulation}
Let $u$ satisfy the hypotheses either of Theorem~\ref{th:pl2} or of Theorem  \ref{th:pg2}, $p = \frac{n}{s}$. Set $w := |\laps{s} u|^{p-2}\laps{s}u$. Then $w$ satisfies
\begin{equation}\label{schro}
 \laps{s}w^i = \Omega_{ij} w^j + E_i(w),
\end{equation}
where $\Omega_{ij} = -\Omega_{ji}$ belongs to $L^p(\R^n)$ (for $p \leq 2$) or to $L^{(p,2)}(\R^n)$ (for $p > 2$).

Moreover, $E_i$ is so that for any $\eps > 0$ there exists a radius $R = R(\eps)$ and a $K \in \N$ so that for any $k_0 \in \N$, $k_0 > K$, any $x_0 \in \R^n$ and for any radius $r \in (0,2^{-k_0}R)$ it holds that for any $\varphi \in C_c^\infty(B(x_0,r))$
\[
 \int_{\R^n} E_i(w)\varphi \aleq \eps\, \brac{\|\varphi\|_{\infty} + \|\laps{s} \varphi\|_{(p,2)}}\ \left(\|w\|_{(p',\infty),B(x_0,2^{k_0} r)}+ \sum_{k=k_0}^\infty 2^{-k\sigma}\|w\|_{(p',\infty),B(x_0,2^{k} r)}\right). 
\]
if $p>2$ and
\[
 \int_{\R^n} E_i(w)\varphi \aleq \eps\, \brac{\|\varphi\|_{\infty} + \|\laps{s} \varphi\|_{p}}\ \left(\|w\|_{(p',\infty),B(x_0,2^{k_0} r)}+ \sum_{k=k_0}^\infty 2^{-k\sigma}\|w\|_{(p',\infty),B(x_0,2^{k} r)}\right). 
\]
if $p\le 2$.
Here $\sigma > 0$ is a uniform constant only depending on $s$ and $n$.
\end{proposition}

Then, Theorem~\ref{th:pl2} and Theorem~\ref{th:pg2} follow from the following result on Schr\"odinger-type equations and the Sobolev embedding for Sobolev-Morrey spaces \cite{Adams-1975}.
\begin{proposition}\label{antisympotreg}
If $w \in L^{\frac{n}{n-s}}(\R^n,\R^N)$ is a solution of
\begin{equation}\label{eq:antisymeq}
 \laps{s}w^i = \Omega_{ij}\, w^j + E_i(w) \quad \mbox{in $\R^n$},
\end{equation}
where $\Omega_{ij} = -\Omega_{ji} \in L^{(\frac{n}{s},2)}$ and $E$ is as in Proposition  \ref{pr:reformulation}. Then there exists $\alpha > 0$ so that for every $x_0\in \R^n$  it holds
\begin{equation}\label{eq:morreyest}
 \sup_{x\in B(x_0,\rho)} \rho^{-\alpha} \|w\|_{(\frac{n}{n-s},\infty),B(x,\rho)} < \infty.
\end{equation}
\end{proposition}
Proposition~\ref{antisympotreg} implies in particular, that solutions of 
\[
 \laps{s}w^i = \Omega_{ij} w^j 
\]
improve their integrability when $\Omega \in L^{(\frac{n}{s},1)}$ without any antisymmetry assumption. Indeed, then $\Omega_{ij} w^j$ satisfies the conditions of $E_i$. This special case is related  to the Lipschitz regularity of solutions of 
\[
 \div(|\nabla u|^{n-2}\nabla u) = \Omega |\nabla u|^{n-2}\nabla u
\]
under the assumption that $\Omega \in L^{(n,1)}$, which was proven by Duzaar and Mingione, \cite{Duzaar-Mingione-2010}.

Let us also remark, that in the local case, i.e. for $s = 2$ and $n=2$, the assumption of Proposition~\ref{antisympotreg} are not optimal: Rivi\`ere showed in \cite{Riviere-2011} that in that case $\Omega_{ji} \in L^{(\frac{n}{s},\frac{n}{s})}$ suffices to improve integrability. Nevertheless, observe that for $n=1$ and $s = \frac{1}{2}$ we recover the regularity Theorem by Rivi\`{e}re and the first author \cite{DLR2}. Also, for $\frac{n}{s} < 2$ our assumptions are weaker than $\Omega \in L^{\frac{n}{s}}(\R^n)$.\par
The paper is organized as follows. In Section~\ref{defnot} we introduce some preliminary definitions and notations. Section~\ref{s:proofpr:reformulation} is devoted to the proof of Proposition~\ref{pr:reformulation}.
In Section~\ref{s:goodgauge} we show how to perform a {\em change of gauge} in a system of the form \eqref{schro}. In Section~\ref{s:proofantisympotreg} we prove  Proposition \ref{antisympotreg}.

\section{Preliminaries: function spaces and the fractional Laplacian}\label{defnot}
In this Section we introduce some  notations and definitions that are used in the paper.
 
For $n\ge 1$,  we denote respectively by ${\mathcal{S}}(\R^n)$ and  ${\mathcal{S}}^{\prime}(\R^n)$ the spaces of Schwartz functions and tempered distributions. 

Given a function $v$ we will denote either by
  $\hat v$ or by  ${\mathcal{F}}[v]$ the Fourier Transform of  $v$ :
  $$\hat v(\xi)={\mathcal{F}}[v](\xi)=\int_{\R^n}v(x)e^{-i \langle \xi, x\rangle }\,dx\,.$$
   
  We introduce the following topological subspace of ${\mathcal{S}}(\R^n):$
  $$
  {\mathcal{Z}}(\R^n)=\{\varphi\in {\mathcal{S}}(\R^n): (D^{\alpha}{\mathcal{F}}[v])(0)=0, ~~\mbox{ for every multi-index $ \alpha$}\}.$$
  Its topological dual $  {\mathcal{Z}}'(\R^n)$ can be identified with the quotient  space ${\mathcal{S}}'(\R^n)/{\mathcal{P}}(\R^n)$
  where ${\mathcal{P}}(\R^n)$ is the collection of all polynomials, (see e.g. \cite{Tr83}).

 Given $q>1$ and $s\in\R$ we also set
  
 $$
 \dot  H^{s,q}(\R^n):=\{v\in {\mathcal{Z}}^{\prime} (\R^n):~~{\mathcal{F}}^{-1}[|\xi|^s{\mathcal{F}}[v]]\in L^q(\R^n)\}\,.
 $$
For a submanifold  ${\mathcal{N}}$ of $\R^m$  we can define   
 $$ \dot{H}^{s,q}(\R^n,{\mathcal{N}})=\{u\in\dot  H^{s,q}({\R}^n,\R^m):~~u(x)\in {\mathcal{N}}, {\rm a.e.}\}\,.
 $$
 
Finally we denote    ${\mathcal{H}}^1(\R^n)$  the homogeneous Hardy Space in $\R^n$.    \par
  
  We recall that  if $sp=n$ then
 \begin{equation} 
 \dot{H}^{s,p}(\R^n)\hookrightarrow BMO({\R^n}),\end{equation}
 where $BMO({\R^n})$ is the space of bounded mean oscillation dual to ${\mathcal{H}}^1(\R^n)$.
 \par 
 The $s$-fractional  Laplacian of a function  $u\colon\R^n\to\R$ is defined as a pseudo differential operator of symbol $|\xi|^{2s}$ :
 \begin{equation}\label{fract}
 \widehat { ( {-\Delta})^{s}u}(\xi)=|\xi|^{2s} \hat u(\xi)\,.
 \end{equation}
 For every $\sigma\in(0,n)$ we denote by $I^{\sigma}$ the Riesz Potential, that is
 $$I^{\sigma}f(x):=c_s\int_{\R^n}\frac{f(z)}{|x-z|^{n-s} }dz.$$
 Finally we introduce the definition of Lorentz spaces (see for instance Grafakos's monograph \cite{Gra1} for a complete presentation of such spaces).
For  $1\le p<+\infty,   1\le q\le +\infty$, the Lorentz space  $L^{(p,q)}(\R^n)$  is the set 
     of measurable functions satisfying
$$
 \left\{\begin{array}{ll}
\int_{0}^{+\infty}(t^{1/p}f^*(t))^{q}\frac{dt}{t}<+\infty, & ~\mbox{if $q<\infty,~p<+\infty$}\\[5mm]
\sup_{t>0}t^{1/p}f^*(t)<\infty & ~\mbox{if $q=\infty,~p<\infty$}\,,\end{array}\right.
 $$
where $f^*$ is the decreasing rearrangement of $|f|\,.$\par
We observe that $L^{p,\infty}(\R^n)$ corresponds to the weak $L^p$ space.
 Moreover for $1<p<+\infty, 1\le q\le\infty$  the dual space of $L^{(p,q)}$ is $L^{(\frac{p}{p-1},\frac{q}{q-1})}$ if $q>1$ and it is 
 $L^{(\frac{p}{p-1},\infty)}$ if $q\le 1\,.$
 \par
Let us define
$$\dot  H^{s,(p,q)}(\R^n)=\{ v\in {\mathcal{Z}}^{\prime} (\R^n):~~{\mathcal{F}}^{-1}[|\xi|^s{\mathcal{F}}[v]]\in L^{(p,q)}(\R^n)\}\,.$$
 
 In the sequel we will   often use
  the   H\"older inequality in the  Lorentz spaces: if $f\in L^{p_1,q_1}, g\in L^{p_2,q_2}$,  with $1\le p_1,p_2,q_1,q_2\le +\infty$. Then
$fg\in L^{r,s},$ with $r^{-1}={p_1}^{-1}+{p_2}^{-1}$ and $s^{-1}={q_1}^{-1}+{q_2}^{-1}\,,$ (see for instance \cite{Gra1}).\par

To conclude we introduce some basic notation.

    $B(\bar x,r)$ is the ball of radius $r$ and centered at $\bar x$. If $\bar x=0$ we simply write
  $B_r$\,. If $x,y\in\R^n,$ $x\cdot y$ is the scalar product between $x,y$\,.

 Given a {\em multiindex}  
$\alpha=(\alpha_1,\ldots,\alpha_n)$,  where $\alpha_i$ is a
nonnegative integer, we denote by  $|\alpha|=\alpha_1+\ldots+\alpha_n\,$ the order of $\alpha$.  \par

 Given $q>1$ we denote by $q^\prime$ the conjugate of $q$: $q^{-1}+{q^{\prime}}^{-1}=1\,.$\par
  In the sequel we will often use the symbols  $a \aleq b$ and $a \aeq b$ instead of $a \le C\, b$ and $C^{-1} b \leq a \leq C b$, respectively, whenever the multiplicative constants $C$ appearing in the estimates are not relevant for the computations and therefore they are omitted.
 
\section{Rewriting the Euler-Lagrange equations: Proof of Proposition~\ref{pr:reformulation}}\label{s:proofpr:reformulation}
For a fixed manifold $\mathcal{N}$ and $p \in \mathcal{N}$ we denote by 
$\Pi(p)$ the projection onto the tangent plane $T_p \mathcal{N}$, and by $\Pi^\perp(p) = I-\Pi(p)$ the projection onto the normal space $(T_p \mathcal{N})^\perp$.

For $s > 0$ we first introduce the following three-term commutator
\begin{equation}\label{hs}
H_s(f,g)=\laps{s} (fg)-\laps{s} f g-f\laps{s} g.
\end{equation}
Such a commutator has been used for the first time in \cite{DLR1} in the case $s=\frac{1}{2}$ in the context of $1/2$-harmonic maps (see also \cite{Schikorra-n2}).  
It represents the  error term of the Leibniz rule for $\laps{s}.$  We recall here some estimates of \eqref{hs} for general $s > 0$ that we will use in the sequel, ( see e.g. \cite{Schikorra-eps,Lenzmann-Schikorra-2016}).
\begin{lemma}\label{reghs}
Let $s \in (0,1]$, For any $t \in (0,s)$, $p,p_1,p_2 \in (1,\infty)$, $q,q_1,q_2 \in [1,\infty]$ such that
\[
 \frac{1}{p} = \frac{1}{p_1} + \frac{1}{p_2}, \quad \frac{1}{q} = \frac{1}{q_1} + \frac{1}{q_2},
\]
it holds that
\[
\|H_{s} (f, \varphi)\|_{L^{(p,q)}(\R^n)} \aleq \|\laps{s-t} f\|_{L^{(p_1,q_1)}(\R^n)}\ \|\laps{t} \varphi\|_{L^{(p_2,q_2)}(\R^n)},
\]
\end{lemma}
\begin{lemma}\label{reghs2}
Let $s \in (0,1]$, $p \in (1,\infty)$, $p' = \frac{p}{p-1}$, $q \in [1,\infty]$, $q' = \frac{q}{q-1} \in [1,\infty]$. Then for any $a,b \in C_c^\infty(\R^n)$,
\begin{equation}\label{eq:Hsbmoest}
\int_{\R^n} H_s(a,b)\, \laps{s} \varphi \aleq [\varphi]_{BMO}\, \|\laps{s} a\|_{L^{(p,q)}(\R^n)}\|\laps{s} b\|_{L^{(p',q')}(\R^n)}.
\end{equation}
In particular, by the duality of Hardy-space $\mathcal{H}^1$ and BMO,
\[
\|\laps{s} \brac{ H_s(a,b)}\|_{\mathcal{H}^1} \aleq \|\laps{s} a\|_{L^{(p,q)}}\|\laps{s} b\|_{L^{(p',q')}}.
\]
\end{lemma}
We will recall the following result
\begin{lemma}[Coifman-Rochberg-Weiss \cite{Coifman-Rochberg-Weiss-76}]\label{reghs3}
  For any smooth and compactly supported
$f,g\in C^{\infty}_c(\R^n)$ and any $i = 1, \ldots,n$ we define the commutator
  $$ [\Rz_i,f](g)=\Rz_i(fg)-f\Rz_i(g)$$
 
Then for $p>1$ there is constant $C >0$ (depending on $p,n$) such that
$$\|\Rz_i,f](g)\|_{L^p(\R^n)}\leq C\|f\|_{BMO(\R^n)}\|g\|_{L^p(\R^n)}.$$
  \end{lemma}
We will use the following  extension of Lemma \ref{reghs3}.
\begin{lemma}[Theorem 6.1 in \cite{Lenzmann-Schikorra-2016}]\label{reghs4}
Let $s \in (0,1]$ and $p \in (1,\infty)$ and $q,q_1,q_2 \in [1,\infty]$ with $\frac{1}{q} = \frac{1}{q_1} + \frac{1}{q_2}$. Then, for $f,g \in C_c^\infty(\R^n)$,
and for $p_1,p_2,p \in (1,\infty)$, $\frac{1}{p_1} + \frac{1}{p_2} = \frac{1}{p}$, $\sigma \in [s,1)$,
\begin{equation}\label{eq:lapsgfglapsfintermediate}
\| [\laps{s}, g](f)\|_{L^{(p,q)}(\R^n)} \aleq \|\laps{\sigma} g\|_{L^{(p_1,q_1)}(\R^n)}\ \|\lapms{\sigma-s} f \|_{L^{(p_2,q_2)}(\R^n)}.
\end{equation}
\eqref{eq:lapsgfglapsfintermediate} remains valid if one replaces $\laps{s}$ by $\Rz_i \laps{s}$, where $\Rz_i$ is the $i$th Riesz transform.

Also, for $p_1,p_2,p \in (1,\infty)$, $\frac{1}{p_1} + \frac{1}{p_2} = \frac{1}{p}$, $\sigma \in [0,1)$,
\begin{equation}\label{eq:rieszgfintermediate}
 \| [\Rz_i, g] f \|_{L^{(p,q)}(\R^n)} \aleq \|\laps{\sigma} g\|_{L^{(p_1,q_1)}(\R^n)}\ \|\lapms{\sigma} f\|_{L^{(p_2,q_2)}(\R^n)}.
\end{equation}
\end{lemma}
For a map $u: \R^n \to \mathcal{N}$ any derivative $\partial_\alpha u$ is a tangential vector, i.e. $\partial_\alpha u\in T_u \mathcal{N}$. In particular, $\Pi^\perp(p) \nabla u = 0$. If we replace the gradient $\nabla u$ by $\laps{s} u$ there is no reason for this to be true. However, a certain tangential inclination of $\laps{s} u$ can be measured in the following sense.
\begin{lemma}\label{la:projectionlapu}
Assume that $u \in \dot{H}^{s,p}(\R^n,\mathcal{N})$, where $p=\frac{n}{s} \in (1,\infty)$. Then
\begin{equation}\label{est1}
 \|\Pi^\perp(u) \laps{s} u\|_{(p,q)} \aleq \|\laps{s} u\|_{(p,q_1)} \|\laps{s} u\|_{(p,q_2)},
\end{equation}
whenever $\frac{1}{q} = \frac{1}{q_1} + \frac{1}{q_2}$. 

Moreover, for $p \leq 2$
\begin{equation}\label{est2}
 \||\laps{s} u|^{p-2}\Pi^\perp(u) \laps{s} u\|_{p'} \aleq \|\laps{s} u\|_{(p,\infty)}^{p-1}\ \|\laps{s} u\|_{(p,p)}^{p-1}
\end{equation}

Also the localized versions of the above estimates hold: for some $\sigma =\sigma(s)> 0$, for every $k_0 \in \N$,
\[
\begin{split}
 \|\Pi^\perp(u) \laps{s} u\|_{(p,q),B(x_0,r)} \aleq& \brac{\|\laps{s} u\|_{(p,q_1),B(x_0,2^{k_0}r)} + \sum_{k=k_0}^\infty 2^{-k\sigma} \|\laps{s} u\|_{(p,\infty),B(x_0,2^{k}r)}}\\
 &\cdot\brac{\|\laps{s} u\|_{(p,q_2),B(x_0,2^{k_0}r)} + \sum_{\ell =k_0}^\infty 2^{-\ell \sigma} \|\laps{s} u\|_{(p,\infty),B(x_0,2^{\ell}r)}}.
\end{split}
 \]
and 
\[
\begin{split}
  \||\laps{s} u|^{p-2}\Pi^\perp(u) \laps{s} u\|_{p'} \aleq& \brac{\|\laps{s} u\|^{p-1}_{(p,\infty),B(x_0,2^{k_0}r)} + \sum_{k=k_0}^\infty 2^{-k\sigma} \|\laps{s} u\|^{p-1}_{(p,\infty),B(x_0,2^{k}r)}}\\
 &\cdot\brac{\|\laps{s} u\|^{p-1}_{(p,p),B(x_0,2^{k_0}r)} + \sum_{\ell =k_0}^\infty 2^{-\ell \sigma} \|\laps{s} u\|^{p-1}_{(p,\infty),B(x_0,2^{\ell}r)}},
\end{split}
 \]
\end{lemma}
\begin{proof} The localization arguments are by now standard, we only indicate how to prove the global estimates.

The estimate \eqref{est1} follows for $s \in (0,1]$ from
\begin{equation}\label{eq:Piperplapsu}
 |\Pi^\perp(u) \laps{s} u(x)|\aleq |H_{s}(u,u)|,
\end{equation}
(see e.g. \cite[Lemma E.1.]{Schikorra-2017},  and also Proposition 4.1 in \cite{DLLR-16}, for related properties) and by applying Lemma \ref{reghs}. For $s \geq 1$ we use that $\laps{s} = \laps{s-1}\Rz_{\alpha}\partial_\alpha$, and thus $\Pi^\perp(u) \partial_\alpha u = 0$ implies
\[
 \Pi^\perp(u) \laps{s} u = [\Pi^\perp(u),\laps{s-1}\Rz_{\alpha}](\partial_\alpha u)= \Pi^\perp(u)\laps{s-1}\Rz_{\alpha}(\partial_\alpha(u))-\laps{s-1}\Rz_{\alpha}( \Pi^\perp(u)\partial_\alpha(u))
\]
The estimate then follows from Lemma~\ref{reghs4}.

For the second estimate \eqref{est2}, assume that $p \in (1,2]$, and observe $\Pi \in L^\infty(\mathcal{N},\R^N)$ implies that pointwise
\[
 ||\laps{s} u|^{p-2}\Pi^\perp(u) \laps{s} u|\aleq |\laps{s} u|^{p-1}.
\]
Moreover, in view of \eqref{eq:Piperplapsu}, for $s \in (0,1)$,
\[
 ||\laps{s} u|^{p-2}\Pi^\perp(u) \laps{s} u|\aleq |\laps{s} u|^{p-2} |H_{s}(u,u)|.
\]
Pointwise interpolating these two estimates, for any $\beta \in [0,1]$
\[
 ||\laps{s} u|^{p-2}\Pi^\perp(u) \laps{s} u|\aleq |\laps{s} u|^{\beta(p-1)} |\laps{s} u|^{(1-\beta)(p-2)} |H_{s}(u,u)|^{1-\beta}.
\]
Since $p \in (1,2]$, set $\beta = 2-p \in [0,1]$. Thus,
\[
 ||\laps{s} u|^{p-2}\Pi^\perp(u) \laps{s} u|\aleq |H_{s}(u,u)|^{p-1}.
\]
Thus,
\[
  \||\laps{s} u|^{p-2}\Pi^\perp(u) \laps{s} u\|_{p'} \aleq \|H_{s}(u,u)\|_{p}^{p-1} \aleq \|\laps{s} u\|_{(p,\infty)}^{p-1}\ \|\laps{s} u\|_{(p,p)}^{p-1}.
\]
The case $s \geq 1$ follows once again from Lemma~\ref{reghs4}.
\end{proof}

We have all the ingredients for Proposition~\ref{pr:reformulation}.
\begin{proof}[Proof of Proposition~\ref{pr:reformulation}]
Recall the definition of projections $\Pi(u)$ and $\Pi^\perp(u)$ above. Observe that these are symmetric matrices.

Also observe that for $p < 2$, $\|\laps{s} u\|_{(p,2)} \aleq \|\laps{s} u\|_{p}$.

From \eqref{eq:eleq} and $\Pi(u) + \Pi^\perp(u) = Id$ we have for any $\varphi \in C_c^\infty(\R^n,\R^N)$
\[
\begin{split}
&\int_{\R^n}|\laps{s} u|^{p-2}\laps{s} u\cdot \laps{s} \varphi \\
=&\int_{\R^n}|\laps{s} u|^{p-2}\Pi(u)\laps{s} u\cdot \laps{s}(\Pi^\perp(u) \varphi)\\
&+\int_{\R^n}|\laps{s} u|^{p-2}\Pi^\perp(u)\laps{s} u\cdot \laps{s}(\Pi^\perp(u)\varphi).
\end{split}
\]
 
Then with $\Pi(u) \Pi^\perp (u) = 0$ we find
\[
\begin{split}
&\int_{\R^n}|\laps{s} u|^{p-2}\laps{s} u\cdot \laps{s} \varphi \\
=&\int_{\R^n}|\laps{s} u|^{p-2}\laps{s}\Pi^\perp(u)\ \Pi(u)\laps{s} u\cdot  \varphi\\
&+\int_{\R^n}|\laps{s} u|^{p-2}\Pi(u)\laps{s} u\cdot H_s(\Pi^\perp(u), \varphi)\\
&+\int_{\R^n}|\laps{s} u|^{p-2}\Pi^\perp(u)\laps{s} u\cdot \laps{s}(\Pi^\perp(u)\varphi).
\end{split}
\]
Now we define $\Omega \in L^{p}(\R^n,so(N))$,
\[
 \Omega := \laps{s}\Pi^\perp(u)\ \Pi(u) - \Pi(u)\laps{s}\Pi^\perp(u).
\]
Observe that for $p >2$ the extra assumption \eqref{eq:extraassumption} implies $\Omega \in L^{(p,2)}(\R^n,so(N))$.
Then we have
\[
\begin{split}
&\int_{\R^n}|\laps{s} u|^{p-2}\laps{s} u\ \laps{s} \varphi \\
=&\int_{\R^n}|\laps{s} u|^{p-2}\Omega\ \laps{s} u\, \varphi\\
&+\int_{\R^n}|\laps{s} u|^{p-2}\Pi(u) \laps{s}\Pi^\perp(u)\ \laps{s} u\ \, \varphi\\
&+\int_{\R^n}|\laps{s} u|^{p-2}\Pi(u)\laps{s} u\ H_s(\Pi^\perp(u), \varphi)\\
&+\int_{\R^n}|\laps{s} u|^{p-2}\Pi^\perp(u)\laps{s} u\ \laps{s}(\Pi^\perp(u)\varphi)\\
\end{split}
\]
And again by $\Pi(u) \Pi^\perp(u) = 0$,
\[
\begin{split}
&\int_{\R^n}|\laps{s} u|^{p-2}\laps{s} u\ \laps{s} \varphi \\
=&\int_{\R^n}|\laps{s} u|^{p-2}\Omega\ \laps{s} u\, \varphi + \int_{\R^n} E\ \varphi\\
\end{split}
\]
where
\[
 \begin{split}
\int_{\R^n} E\ \varphi  
:=&-\int_{\R^n}|\laps{s} u|^{p-2} H_s(\Pi(u),\Pi^\perp(u))\ \laps{s} u\ \, \varphi\\
&+\int_{\R^n}|\laps{s} u|^{p-2}\Pi(u)\laps{s} u\ H_s(\Pi^\perp(u), \varphi)\\
&-\int_{\R^n}|\laps{s} u|^{p-2}\laps{s}\Pi(u) \Pi^\perp(u)\ \laps{s} u\ \, \varphi\\
&+\int_{\R^n}|\laps{s} u|^{p-2}\Pi^\perp(u)\laps{s} u\ \laps{s}(\Pi^\perp(u)\varphi).
 \end{split}
\]
Now
\[
\begin{split}
 &\int_{\R^n}|\laps{s} u|^{p-2} H_s(\Pi(u),\Pi^\perp(u))\ \laps{s} u\ \, \varphi\\
 \aleq&\|\laps{s} u\|_{(p,\infty)}^{p-1}\ \|H_s(\Pi(u),\Pi^\perp(u))\|_{(p,1)}\ \|\varphi\|_{\infty} \\
\end{split}
 \]
And by the fractional Leibniz rule, see \cite{Lenzmann-Schikorra-2016},
\[
 \|H_s(\Pi(u),\Pi^\perp(u))\|_{(p,1)} \aleq \|\laps{s} u\|_{(p,2)}^2.
\]
Thus,
\[
\begin{split}
 &\int_{\R^n}|\laps{s} u|^{p-2} H_s(\Pi(u),\Pi^\perp(u))\ \laps{s} u\ \, \varphi\\
 \aleq&\|\laps{s} u\|_{(p,\infty)}^{p-1}\ \|\laps{s} u\|_{(p,2)}^2\ \|\varphi\|_{\infty}.
\end{split}
 \]
This is the crucial point where our assumption $\|\laps{s} u\|_{(p,2)} < \infty$ enters (which is only a nontrivial assumption for $p > 2$).
In the same spirit,
\[
\begin{split}
&\int_{\R^n}|\laps{s} u|^{p-2}\Pi(u)\laps{s} u\ H_s(\Pi^\perp(u), \varphi)\\
 \aleq&\|\laps{s} u\|_{(p,\infty)}^{p-1}\ \|\laps{s} u\|_{(p,2)}\ \|\laps{s} \varphi\|_{(p,2)}.
\end{split} 
\]
The remaining terms of $E$ can be estimated by Lemma~\ref{la:projectionlapu}.\par
For $p \leq 2$:
\[
\begin{split}
 &|\int_{\R^n}|\laps{s} u|^{p-2}\laps{s}\Pi(u) \Pi^\perp(u)\ \laps{s} u\ \, \varphi|\\
\aleq& \|\laps{s}u\|_{(p,\infty)}^{p-1} \|\laps{s}u\|_{p}^{p-1}\, \|\laps{s}\Pi(u)\|_{p}\, \|\varphi\|_{\infty}\\
\aleq& \|\laps{s}u\|_{(p,\infty)}^{p-1} \|\laps{s}u\|_{p}^{p}\, \|\varphi\|_{\infty}
 \end{split}
\]
and
\[
\begin{split}
 &\int_{\R^n}|\laps{s} u|^{p-2}\Pi^\perp(u)\laps{s} u\ \laps{s}(\Pi^\perp(u)\varphi)\\
 \aleq&\|\laps{s}u\|_{(p,\infty)}^{p-1} \|\laps{s}u\|_{p}^{p-1}\, \|\laps{s}(\Pi^\perp(u) \varphi)\|_{p}\\
 \aleq&\|\laps{s}u\|_{(p,\infty)}^{p-1} \|\laps{s}u\|_{p}^{p-1}\, \brac{\|\laps{s} u\|_{p} \|\varphi\|_\infty + \|\laps{s} \varphi\|_{p} + \|\laps{s} u\|_{p} \|\laps{s} \varphi\|_p}
 \end{split}
\]
where to estimate $ \|\laps{s}(\Pi^\perp(u) \varphi)\|_{p}$ we use the fact that
\[\begin{split}&\laps{s}(\Pi^\perp(u) \varphi)=H_s(\Pi^\perp(u),\varphi)+\laps{s}(\Pi^\perp(u))\varphi+\Pi^\perp(u)\laps{s}\varphi\end{split}\]
Lemma \ref{reghs} and Sobolev embeddings.\par
For $p > 2$:
\[
\begin{split}
 &|\int_{\R^n}|\laps{s} u|^{p-2}\laps{s}\Pi(u) \Pi^\perp(u)\ \laps{s} u\ \, \varphi|\\
\aleq& \|\laps{s}u\|_{(p,\infty)}^{p-2} \|\laps{s}\Pi(u)\|_{(p,\infty)}\, \|\Pi^\perp(u)\ \laps{s} u\|_{(p,1)}\, \, \|\varphi\|_{\infty}\\
\aleq& \|\laps{s}u\|_{(p,\infty)}^{p-1} \|\laps{s}u\|_{(p,2)}^{2}\, \|\varphi\|_{\infty}.
 \end{split}
\]
and
\[
\begin{split}
 &\int_{\R^n}|\laps{s} u|^{p-2}\Pi^\perp(u)\laps{s} u\ \laps{s}(\Pi^\perp(u)\varphi)\\
\aleq& \|\laps{s}u\|_{(p,\infty)}^{p-2} \|\laps{s}(\varphi \Pi^{\perp}(u))\|_{(p,2)}\, \|\Pi^\perp(u)\ \laps{s} u\|_{(p,2)}.\\
\aleq& \|\laps{s}u\|_{(p,\infty)}^{p-1} \|\laps{s}u\|_{(p,2)}\, \brac{\|\laps{s} u\|_{(p,2)} \|\varphi\|_\infty + \|\laps{s} \varphi\|_{(p,2)} + \|\laps{s} u\|_{(p,2)} \|\laps{s} \varphi\|_{(p,2)}}.
 \end{split}
\]
Therefore we get for $p\le 2$
\[
 \int_{\R^n} E\ \varphi  \aleq \|\laps{s} u\|_{(p,\infty)}^{p-1}\, (1+\|\laps{s} u\|_{p})\, \|\laps{s} u\|_{(p,2)}^{p-1}\ \brac{\|\varphi\|_{\infty}+\|\laps{s} \varphi\|_{p}}.
\]
 and for $p>2$
 \[
 \int_{\R^n} E\ \varphi  \aleq \|\laps{s} u\|_{(p,\infty)}^{p-1}\, (1+\|\laps{s} u\|_{p})\, \|\laps{s} u\|_{(p,2)}\ \brac{\|\varphi\|_{\infty}+\|\laps{s} \varphi\|_{(p,2)}}.
\]

Consequently  if we assume that $\varphi \in C_c^\infty(B(x_0,r))$, the above estimates  can be localized and we   find for $p \le  2$ 
\begin{eqnarray*}
 \int_{\R^n} E\ \varphi & \aleq& C(1+\|\laps{s} u\|_{p,\R^n}) \brac{\|\laps{s} u\|_{(p,\infty),B(x_0,2^{k_0} r)}^{p-1}+\sum_{k=k_0}^\infty 2^{-\sigma k}\|\laps{s} u\|_{(p,\infty),B(x_0,2^{k} r)}^{p-1}}\\
 ~~~&& \brac{\|\laps{s} u\|_{p,B(x_0,2^{k_0}r)}^{p-1} + \sum_{k=k_0}^\infty 2^{-\sigma k}\|\laps{s} u\|_{p,B(x_0,2^{k}r)}^{p-1}}\ \brac{\|\varphi\|_{\infty}+\|\laps{s} \varphi\|_{p}},
\end{eqnarray*}
and for $p > 2$ if we additionally assume \eqref{eq:extraassumption}, 
\begin{eqnarray*}
 \int_{\R^n} E\ \varphi & \aleq& C(1+\|\laps{s} u\|_{(p,2)})\ \brac{\|\laps{s} u\|_{(p,\infty),B(x_0,2^{k_0} r)}^{p-1}+\sum_{k=k_0}^\infty 2^{-\sigma k}\|\laps{s} u\|_{(p,\infty),B(x_0,2^{k} r}^{p-1}}\,\\
 ~~~&& \brac{\|\laps{s} u\|_{(p,2),B(x_0,2^{k_0}r)} + \sum_{k=k_0}^\infty 2^{-\sigma k}\|\laps{s} u\|_{(p,2),B(x_0,2^{k}r)}}\ \brac{\|\varphi\|_{\infty}+\|\laps{s} \varphi\|_{(p,2)}}.
\end{eqnarray*}
For all $k_0$ sufficiently large and $2^{k_0}r$ sufficiently small, we can assume by absolute continuity of the integral that 
\[
 \|\laps{s} u\|_{p,B(x_0,2^{k_0}r)}^{p-1} + \sum_{k=k_0}^\infty 2^{-\sigma k}\|\laps{s} u\|_{p,B(x_0,2^{k}r)}^{p-1}  < \eps.
\]
and under the assumption \eqref{eq:extraassumption} also
\[
 \|\laps{s} u\|_{(p,2),B(x_0,2^{k_0}r)}+ \sum_{k=k_0}^\infty 2^{-\sigma k}\|\laps{s} u\|_{(p,2),B(x_0,2^{k}r)} < \eps
\]
This proves the localized estimate for $E$ and we conclude the proof of Proposition~\ref{pr:reformulation}.
\end{proof}

\section{Construction of a good gauge}\label{s:goodgauge}
The next theorem is an adaption of \cite[Theorem~1.2]{DLR2} of Rivi\`{e}re and the first-named author, which is a choice of a good gauge. It follows the strategy developed by Rivi\`{e}re in \cite{Riviere-2007} 
which was itself inspired by Uhlenbeck’s construction of Coulomb gauges \cite{Uhlenbeck-1982}. For extensions and relations to the moving frame method by H\'elein see also \cite{Schikorra-2010,DLnD,Schikorra-eps,Goldsteins-2017}.

\begin{theorem}[Choice of gauge]\label{th:gauge}
There exists $\eps > 0$ so that the following holds.

Whenever $\Omega \in L^{(\frac{n}{s},2)}(\R^n,so(N))$ satisfies 
\[
 \|\Omega\|_{(\frac{n}{s},2)} < \eps,
\]
then there exists $P \in \dot{H}^{s,(\frac{n}{s},2)}(\R^n,SO(N))$ so that
\[
 \|\laps{s} P + P \Omega\|_{L^{(\frac{n}{s},1)}(\R^n,\R^{N\times N})} \aleq \|\Omega\|_{(\frac{n}{s},2)}.
\]
Moreover,
\[
 \|\laps{s} P\|_{L^{(\frac{n}{s},2)}(\R^n,\R^{N\times N})} \aleq \|\Omega\|_{(\frac{n}{s},2)}.
\]
\end{theorem}
Theorem~\ref{th:gauge} is a consequence of the following
\begin{theorem}\label{th:gauge2}
For any $q \in [1,\infty]$ there exist $\eps > 0$ so that if 
\[
\|\Omega\|_{(\frac{n}{s},q)} < \eps,
\]
then there exists $P \in \dot{H}^{s,(p,q)}(\R^n,SO(N))$ so that 
\begin{equation}\label{eq:Pgauge}
 P^T \laps{s} P - \laps{s} P\, P^T + 2\Omega = 0 \quad \mbox{in $\R^n$}.
\end{equation}
\end{theorem}

\begin{proof}[Proof of Theorem~\ref{th:gauge}]
From \eqref{eq:Pgauge}
\[
 \laps{s} P   + P\Omega  = \frac{1}{2}\brac{\laps{s} P P^T + P \laps{s} P^T}\, P,
\]
that is, since $\laps{s} (PP^T) \equiv \laps{s} I \equiv 0$,
\[
 \laps{s} P   + P\Omega  = \frac{1}{2}\, H_s (P,P^T)\, P,
\]
Since $\laps{s} P \in L^{(\frac{n}{s},2)}(\R^n)$ from the three commutator estimates (see Lemma \ref{reghs}) we find that 
\[
 H_s (P,P^T) \in L^{(\frac{n}{s},1)}(\R^n,\R^N), 
\]
and have consequently shown that 
\[
 \laps{s} P   + P\Omega  \in L^{(\frac{n}{s},1)}(\R^n,\R^N).
\]
\end{proof}

\subsection{Construction of the optimal gauge: Proof of Theorem~\ref{th:gauge2}}
In order to establish \eqref{eq:Pgauge} we adapt the strategy from \cite[Theorem~1.2]{DLR2}. For notational simplicity we prove this theorem only for $L^{(\frac{n}{s},2)}$ (i.e. $q=2$) the case we need.

For the rest of this section fix $1 < q_1, q_2 < \infty$ exponents so that $1 < q_1 < \frac{n}{s} < q_2 < \infty$. 

As in \cite[Proof of Theorem~1.2, Step 4]{DLR2}, by an approximation argument it suffices to prove the claim under the stronger assumption that $\Omega \in L^{q_1}\cap L^{q_2}(\R^n)$ with good estimates. More precisely, for $\eps > 0$
let
\[
     \mathcal{U}_\eps  := \left \{\Omega \in L^{q_1} \cap L^{q_2} (\R^n,so(N)):\|\Omega\|_{(\frac{n}{s},2)} \leq \eps \right \},
    \]
and for constants $\eps, \Theta > 0$ let $\mathcal{V}_{\eps,\Theta} \subset \mathcal{U}_\eps$ be the set where we have the decomposition \eqref{eq:Pgauge} with the estimates
\begin{equation}\label{eq:gauge:p2est}
  \|\laps{s} P\|_{(p,2)} \leq \Theta\, \|\Omega\|_{(p,2)}
\end{equation}
\begin{equation}\label{eq:gauge:q12}
  \|\lapv P\|_{q_1} \leq \Theta \|\Omega\|_{q_1}, \quad \|\lapv P\|_{q_2} \leq \Theta \|\Omega\|_{q_2}.
\end{equation}
That is,
\[
 \mathcal{V}_{\eps,\Theta} := \left \{\Omega \in \mathcal{U}_\eps:\ \begin{array}{c}
                                             \mbox{there exists $P \in \dot{H}^{s,q_1} \cap \dot{H}^{s,q_2}(\R^n,SO(N))$, so that}\\
                                             P-I \in L^{\frac{nq_1}{n-q_1 s}}(\R^n,\R^{N \times N}) \mbox{ and \eqref{eq:gauge:p2est}, \eqref{eq:gauge:q12},}\\
                                             \mbox{and \eqref{eq:Pgauge} holds}.
                                            \end{array}
\right \}
\]
Let us remark a technical detail. The condition $P-I \in L^{\frac{nq_1}{n-q_1 s}}(\R^n,\R^{N \times N})$ corresponds to prescribing Dirichlet data at infinity.   With the definition of homogeneous Sobolev spaces as above, the proof below works also without this Dirichlet assumption which essentially corresponds to a Neumann-type condition at infinity. For our purpose, there is no advantage to either choice.  
We then need to prove the following
\begin{proposition}\label{pr:VepseqUeps}
There exist $\Theta > 0$ and $\eps > 0$ so that $\mathcal{V}_{\eps,\Theta} = \mathcal{U}_\eps$.
\end{proposition}
Proposition~\ref{pr:VepseqUeps} follows from a continuity method, once we show the following four properties
\begin{itemize}
 \item[(i)] $\mathcal{U}_\eps$ is  connected.
 \item[(ii)] $\mathcal{V}_{\eps,\Theta}$ is nonempty.
 \item[(iii)] For any $\eps, \Theta > 0$, $\mathcal{V}_{\eps,\Theta}$ is a relatively closed subset of $\mathcal{U}_\eps$.
 \item[(iv)] There exist $\Theta > 0$ and $\eps > 0$ so that $\mathcal{V}_{\eps,\Theta}$ is a relatively open subset of $\mathcal{U}_\eps$.
\end{itemize}
Property (i) is clear, since $\mathcal{U}_\eps$ is starshaped with center $0$: for any $\Omega \in \mathcal{U}_\eps$ we have $t\Omega \in \mathcal{U}_\eps$ for all $t \in [0,1]$. Property (ii) is also obvious since $P :\equiv I$ is an element of $\mathcal{V}_{\eps,\Theta}$. The closedness property (iii) follows almost verbatim from \cite[Proof of Theorem~1.2, Step~1, p.1315]{DLR2}: there one replaces $\lapv$ by $\laps{s}$, $q$ by $q_2$, $q'$ by $q_1$, and the $L^2$-norm by the $L^{(\frac{n}{s},2)}$-norm (for which we still can use the lower semicontinuity). Observe that a uniform bound of the $L^{q_1}$-norm as in \eqref{eq:gauge:q12} implies by Sobolev embedding in particular a uniform bound $P-I$ in $L^{\frac{nq_1}{n-q_1 s}}(\R^n,\R^{N \times N})$.

The main point is to show the openness property $(iv)$. For this let $\Omega_0$ be arbitrary in $\mathcal{V}_{\eps,\Theta}$, for some $\eps, \Theta > 0$ chosen below. Let $P_0 \in \dot{H}^{s,q_1} \cap \dot{H}^{s,q_2}(\R^n,SO(N))$, $P_0-I \in L^{\frac{nq_1}{n-q_1 s}}(\R^n,\R^{N \times N})$ so that the decomposition \eqref{eq:Pgauge} as well as the estimates \eqref{eq:gauge:p2est}, \eqref{eq:gauge:q12} are satisfied for $\Omega_0$.

We introduce the map
\[
 F(U) := \brac{P_0\, \exp(U)}^{-T}\ \laps{s} (P_0\, \exp(U)) - \laps{s}\brac{P_0\, \exp(U)}^{-T}\  (P_0\, \exp(U))
\]
Observe that for $U \in L^{\frac{nq_1}{n-q_1 s}}\cap \dot{H}^{s,q_1} \cap \dot{H}^{s,q_2}(\R^n,so(N))$,
\[
 P_0\, \exp(U) -I = (P_0 - I)\exp(U) + I-\exp(U) \in L^{\frac{nq_1}{n-q_1 s}}(\R^n,\R^{N \times N}).
\]
Indeed, observe that $U \in L^\infty$ and thus $(P_0 - I)\exp(U) \in L^{\frac{nq_1}{n-q_1 s}}(\R^n,\R^{N \times N})$. Moreover, 
\[
 |I-\exp(U)| \aleq (1+\|U\|_{\infty}) |U|,
\]
and thus $I-\exp(U) \in L^{\frac{nq_1}{n-q_1 s}}(\R^n,\R^{N\times N})$.

As in \cite[Proof of Theorem~1.2, Step~2, p.1316]{DLR2} we can conclude that $F$ is $C^1$ as a map from
\[
F: L^{\frac{nq_1}{n-q_1 s}}\cap \dot{H}^{s,q_1} \cap \dot{H}^{s,q_2}(\R^n,so(N)) \to L^{q_1} \cap L^{q_2}(\R^n,so(N)).
\]
and that we can compute $DF(0)$ as 
\[
 \frac{d}{dt}\Big|_{t=0}F(t\eta) = L(\eta),
\]
where for $\eta \in L^{\frac{nq_1}{n-q_1 s}}\cap \dot{H}^{s,q_1} \cap \dot{H}^{s,q_2}(\R^n,so(N))$,
\[
 L(\eta) := -\eta P_0^T \laps{s} P_0 + \laps{s}(\eta\, P_0^T)P_0+ P_0^T\, \laps{s}(P_0\eta ) -\laps{s}P_0^T\ P_0\eta
\]
In order to use a fixed-point argument for $F$, we need to show that $L$ is an isomorphism.
\begin{lemma}\label{la:invertible}
For any $\Theta > 0$ there exists a $\eps > 0$ so that the following holds for any $\Omega_0$ and $P_0$ as above.

For any $\omega \in L^{q_1}\cap L^{q_2}(\R^n,so(N))$ there exists a unique $\eta \in L^{\frac{nq_1}{n-q_1 s}}\cap \dot{H}^{s,q_1} \cap \dot{H}^{s,q_2}(\R^n,so(N))$ so that 
\[
\omega = L(\eta) 
\]
and for some constant $C = C(\Omega_0,\Theta) > 0$ it holds
\[
\|\eta\|_{L^{\frac{nq_1}{n-q_1 s}}} + \|\laps{s}\eta\|_{L^{q_1}(\R^n)} +\|\laps{s}\eta\|_{L^{q_2}(\R^n)} \leq C \brac{\|\omega\|_{L^{q_1}(\R^n)} +\|\omega\|_{L^{q_2}(\R^n)}}
\]
\end{lemma}
\begin{proof}
We follow the strategy of \cite[Lemma 4.1]{DLR2}. First we find $\eta$ in some $\dot{H}^{s,r}$ for $r \in (q_1,\frac{n}{s})$ and then that it belongs to the right spaces.

{\bf Step 1:} For $\eta \in L^{\frac{nq_1}{n-q_1 s}}\cap \dot{H}^{s,q_1} \cap \dot{H}^{s,q_2}(\R^n,so(N))$ we rewrite
\[
 L(\eta) = 2\laps{s} \eta + H(\eta)
\]
where
\[
\begin{split}
 H(\eta) :=& \eta \brac{-P_0^T \laps{s} P_0  + \laps{s}P_0^T\, P_0} + \brac{P_0^T\, \laps{s} P_0-\laps{s}P_0^T\ P_0}\eta\\
 &+ H_s(\eta,P_0)P_0 + P_0^T\, H_s(P_0,\eta).
\end{split}
 \]
In particular, for any $r \in (1,\frac{n}{s})$, by H\"older's inequality,
\[
 \|H(\eta) \|_{L^{(r,2)}(\R^n)} \aleq \|\eta\|_{L^{(\frac{rn}{n-rs},\infty)}(\R^n)}\, \|\laps{s} P_0\|_{L^{(\frac{n}{s},2)}(\R^n)}  + \|H_s(\eta,P_0)\|_{L^{(r,2)}(\R^n)}.
\]
By Sobolev embedding,
\[
 \|\eta\|_{L^{(\frac{rn}{n-rs},\infty)}(\R^n)} \aleq \|\laps{s} \eta\|_{L^{(r,\infty)}(\R^n)}.
\]
By the three-commutator estimates,
\[
  \|H_s(\eta,P_0)\|_{L^{(r,2)}(\R^n)} \aleq \|\laps{s} \eta\|_{L^{(r,\infty)}(\R^n)}\ \|\laps{s} P_0\|_{(r,2)}.
\]
Consequently, in view of \eqref{eq:gauge:p2est},
\[
 \|H(\eta) \|_{L^{(r,2)}(\R^n)}  \leq C\, \Theta\, \eps\, \|\laps{s} \eta\|_{L^{(r,\infty)}(\R^n)} \leq C\, \Theta\, \eps\, \|\laps{s} \eta\|_{L^{(r,2)}(\R^n)}
\]
Choosing $\eps$ small enough (depending on $\Theta$), we obtain that $L(\eta)$ is invertible as a map from $L^{(r,2)}(\R^n,so(N))$ to $L^{(\frac{rn}{n-rs},2)}\cap \dot{H}^{s,(r,2)}(\R^n,so(N))$, whenever $r \in (q_1,\frac{n}{s})$.

{\bf Step 2}: For given $\omega \in L^{q_1}\cap L^{q_2}(\R^n,so(N))$ and $r_0 \in (q_1,\frac{n}{s})$ let $\eta \in L^{(\frac{r_0n}{n-r_0s},2)}\cap \dot{H}^{s,(r_0,2)}(\R^n,so(N))$ so that 
\[
 \omega = L(\eta).
\]
We will show that $\laps{s} \eta \in L^{q_2}(\R^n)$. Indeed, since $\laps{s} P_0 \in L^{q_2}(\R^n)$, we can estimate for $\frac{1}{t_1} = \frac{1}{q_2}- \frac{s}{n}+ \frac{1}{r_0}$
\[
 \|H(\eta) \|_{L^{(t_1,2)}(\R^n)} \aleq \|\laps{s} \eta\|_{L^{(r_0,2)}(\R^n)}\, \|\laps{s} P_0\|_{L^{q_2}(\R^n)},
\]
which itself follows from Sobolev embedding and the following estimate from Lemma \ref{reghs}
\[\|H_s(\eta,P_0)\|_{L^{(t_1,2)}(\R^n)}  \aleq \|  \eta\|_{L^{(\frac{r_0n}{n-r_0s},2)}}\, \|\laps{s} P_0\|_{L^{(q_2,\infty)}(\R^n)}.\]
Since
\[
 \laps{s} \eta = \frac{1}{2}\omega - H(\eta)
\]
Now either $t_1 > q_2$, in which case we use that then $L^{(t_1,2)} \cap L^{r_0}(\R^n) \subset L^{q_2}(\R^n)$ (that follows from Step 1) and thus
\[
 \laps{s}\eta \in L^{q_2}\cap L^{r_0}(\R^n). 
\]
Otherwise, we know that $\frac{1}{r_0} -\frac{1}{t_1} = \frac{s}{n} - \frac{1}{q_2} > 0$. In this case we repeat the above argument for $r_1 := t_1$ and find $t_2$ which either is larger than $q_2$ or where $\frac{1}{r_1} -\frac{1}{t_2} = \frac{s}{n} - \frac{1}{q_2} > 0$. Possible repeating this procedure finitely many times we find that eventually some $t_i > q_2$.

{\bf Step 3} It remains to show that $\laps{s} \eta \in L^{q_1}(\R^n)$. Since we already know that $\laps{s} \eta \in L^{(r,2)}\cap L^{q_2}(\R^n)$ for some $r \in (q_1,\frac{n}{s})$ arbitrarily small, we find that $\eta \in L^\infty(\R^n)$.  In particular,
\[
\begin{split}
&\|\eta \brac{-P_0^T \laps{s} P_0  + \laps{s}P_0^T\, P_0} + \brac{P_0^T\, \laps{s} P_0-\laps{s}P_0^T\ P_0}\eta\|_{q_1,\R^n} \\
\aleq &\|\eta\|_{\infty}\ \|\laps{s} P_0\|_{q_1} < \infty.
\end{split}
\]
Moreover, $\laps{s} \eta \in L^{(r,2)} \cap L^{q_2}(\R^n) \subset L^{\frac{n}{s}}(\R^n)$. Thus Lemma \ref{reghs} implies
\[
 \|H_s(\eta,P_0)\|_{L^{q_1}(\R^n)} \aleq \|\laps{s} \eta\|_{L^{\frac{n}{s}}(\R^n)}\ \|\laps{s} P_0\|_{L^{q_1}(\R^n)}.
\]
Consequently,
\[
 \omega - H(\eta) \in L^{q_1}(\R^n),
\]
and thus $\laps{s}\eta \in L^{q_1} \cap L^{q_2}(\R^n)$. Moreover, by interpolation, $\eta \in L^{\frac{nq_1}{n-q_1s}}$. The estimates follow by the above considerations. Lemma~\ref{la:invertible} is proven.
\end{proof}
We continue with the {\bf proof of Proposition \ref{pr:VepseqUeps}.}\par
Thus, by Implicit Function Theorem applied to $F$, if $\eps = \eps(\Theta) > 0$ is chosen small enough, we find for any $\Omega_0 \in \mathcal{V}_{\eps,\Theta}$ some $\delta > 0$ such that for any $\Omega \in \mathcal{U}_\eps$ with
\[
 \|\Omega - \Omega_0 \|_{L^{q_1}(\R^n)}+\|\Omega - \Omega_0 \|_{L^{q_2}(\R^n)} < \delta
\]
we find $P = P_0 e^{U} \in \dot{H}^{s,q_1} \cap \dot{H}^{s,q_2}(\R^n,SO(N))$, so that $P-I \in L^{\frac{nq_1}{n-q_1 s}}(\R^n,\R^{N \times N})$ and \eqref{eq:Pgauge} is satisfied. By continuity of the inverse, we can make $\delta$ possibly smaller to guarantee that
\[
 \|\laps{s} (P-P_0)\|_{q_1,\R^n} \leq \|\Omega\|_{q_1,\R^n}, \quad \|\laps{s} (P-P_0)\|_{q_2,\R^n} \leq \|\Omega\|_{q_2,\R^n}.
\]
 
Observe that this does not right away imply \eqref{eq:gauge:p2est}, \eqref{eq:gauge:q12}. However the above estimate and the fact that $\Omega \in \mathcal{U}_\eps$ imply that for any $\sigma > 0$ we can choose $\eps$ small enough so that 
\[
 \|\laps{s} P\|_{(\frac{n}{s},2),\R^n}\leq \sigma.
\]
The next Lemma shows us that this implies for a small enough choice of $\sigma >0$ that \eqref{eq:gauge:p2est}, \eqref{eq:gauge:q12}
 hold for a uniform constant $\Theta$.
\begin{lemma}\label{la:Pgaugeuniformest}
There exists a $\Theta > 0$ and a $\sigma > 0$ so that whenever $P \in \dot{H}^{s,q_1} \cap \dot{H}^{s,q_2}(\R^n,SO(N))$ and $P-I \in L^{\frac{nq_1}{n-q_1 s}}(\R^n,\R^{N \times N})$ so that \eqref{eq:Pgauge} is satisfied and it holds
\begin{equation}
  \|\laps{s} P\|_{(\frac{n}{s},2),\R^n}\leq \sigma,
\end{equation}
then \eqref{eq:gauge:p2est}, \eqref{eq:gauge:q12} hold.
\end{lemma}
\begin{proof}
In view of \eqref{eq:Pgauge}
\[
 P^T \laps{s} P = \frac{1}{2} H_s(P^T,P) -\Omega. 
\]
In particular, by the three-commutator estimates in Lemma~\ref{reghs}, for a uniform constant $C$ for any $p \in [q_1,q_2]$, $q \in [1,\infty]$,
\[
 \|\laps{s} P\|_{(p,q)} \leq C_1 \|\laps{s} P\|_{(\frac{n}{s},\infty)}\ \|\laps{s} P\|_{(p,q)} + \|\Omega\|_{(p,q)}
\]
Moreover,
\[
 \|\laps{s} P\|_{(\frac{n}{s},\infty)} \leq C_2\, \|\laps{s} P\|_{\frac{n}{s},2} \leq C_2 \sigma
\]
for $\sigma$ small enough we can absorb and find,
\[
 \|\laps{s} P\|_{(p,q)} \leq \frac{1}{1-C_1\, C_2 \sigma} \|\Omega\|_{(p,q)}.
\]
Choosing $\Theta := \frac{1}{1-C_1\, C_2 \sigma}$ we conclude.
\end{proof}
Thus the openness property $(iv)$ is proven, Proposition~\ref{pr:VepseqUeps} is established, and with the approximation argument in \cite[Proof of Theorem~1.2, Step 4]{DLR2} Theorem~\ref{th:gauge2} is proven. \qed

\section{The improved Morrey space estimate: Proof of Proposition~\ref{antisympotreg}}\label{s:proofantisympotreg}
Let $w \in L^{\frac{n}{n-s}}(\R^n,\R^N)$ be a solution of
\[
 \laps{s}w^i = \Omega_{ij}\, w^j + E_i(w) \quad \mbox{in $\R^n$},
\]
where $\Omega_{ij} = -\Omega_{ji} \in L^{(\frac{n}{s},2)}$ and $E$ is as above. 

By absolute continuity of the integral there exists $R > 0$ so that 
\[
 \sup_{x_0 \in \R^n} \|\Omega_{ij}\|_{(\frac{n}{s},2),B(x_0,10R)} < \delta < \eps
\]
for the $\eps > 0$ from Theorem~\ref{th:gauge}, and $\delta$ chosen later. It suffices to prove the claim \eqref{eq:morreyest} in $B(x_0,\rho) \subset B(y_0,R)$, where $y_0 \subset \R^n$ is arbitrary (and the constants will not depend on $y_0$, but may depend on $R$). Let $\eta_{B(y_0,2R)} \in C_c^\infty(B(y_0,10R))$ be the generic smooth cutoff function which is constantly one in $B(y_0,5R)$.
Applying Theorem~\ref{th:gauge} to $\eta_{B(y_0,2R)} \Omega$ we find $P \in \dot{H}^{s,(\frac{n}{s},2)}$ so that 
\[
\|\laps{s} P\|_{(\frac{n}{s},2)} + \|\laps{s} P + P \eta_{B(y_0,2R)}\Omega\|_{L^{(\frac{n}{s},1)}(\R^n,\R^{N\times N})} \aleq \|\eta_{B(y_0,R)} \Omega\|_{(\frac{n}{s},2)}.
\]
We have
\[
\begin{split}
 \laps{s} (Pw) =& -(\laps{s} P + P \Omega)\, w + P E(w) + \brac{\laps{s} (Pw)+\laps{s}P w - P\laps{s} w}\\
 =& -(\laps{s} P+P\eta_{B(y_0,R)} \Omega)\, w-(1-\eta_{B(y_0,R)}) \Omega w+ P E(w)\\
 +& \brac{\laps{s} (Pw)+\laps{s}P w - P\laps{s} w}
 \end{split}
\]\\
In particular, for any $\varphi \in C_c^\infty(B(x_0,\rho))$, for $B(x_0,\rho) \subset B(y_0,R)$, possibly choosing $R$ even smaller for the estimate of $E_i$ to take effect (in the following we write the estimates for the case $n/s > 2$, the case $n/s \leq 2$ is analogous), since $\Omega \varphi = \eta_{B(y_0,2R)}\Omega \varphi$, for all sufficiently large $k_0$, for some $\sigma > 0$
\begin{eqnarray*}
&& \int_{\R^n} Pw\ \laps{s}\varphi \aleq \delta \|w\|_{(\frac{n}{s},\infty),B(x_0,\rho)}\ \|\varphi\|_{\infty}\\
 &&+  \eps\, \brac{\|P\varphi\|_{\infty} + \|\laps{s} (P\varphi)\|_{(\frac{n}{s},2)}}\ \brac{\|w\|_{(\frac{n}{n-s},\infty),B(x_0,2^{k_0} \rho)} + \sum_{k=k_0}^\infty 2^{-k\sigma}\|w\|_{(\frac{n}{n-s},\infty),B(x_0,2^{k }\rho)})} \\
 && + \||w|\,  |H_s(P,\varphi)|\|_{1}.
 \end{eqnarray*}
Firstly,
\[
 \brac{\|P\varphi\|_{\infty} + \|\laps{s} (P\varphi)\|_{(\frac{n}{s},2)}} \aleq \|\laps{s} \varphi\|_{(\frac{n}{s},1)} \brac{{1 + }\|\laps{s} P\|_{(\frac{n}{s},2)}}.
\]
Moreover, by the three commutator estimates and after localization,
\[
 \||w|\,  |H_s(P,\varphi)|\|_{1} \aleq \delta\, \|\laps{s} \varphi\|_{(\frac{n}{s},1)}\ \brac{\|w\|_{(\frac{n}{n-s},\infty),B(x_0,2^{k_0} \rho)} + \sum_{k=k_0}^\infty 2^{-k\sigma}\|w\|_{(\frac{n}{n-s},\infty),B(x_0,2^{k} \rho)}} 
\]
That is, for any $\varphi \in C_c^\infty(B(x_0,\rho))$ so that $\|\laps{s} \varphi \|_{(\frac{n}{s},1)} \leq 1$, we have
\[
\int_{\R^n} Pw\ \laps{s}\varphi \aleq \brac{\eps+\delta} \brac{\|w\|_{(\frac{n}{n-s},\infty),B(x_0,2^{k_0} \rho)} + \sum_{k=k_0}^\infty 2^{-k\sigma} \|w\|_{(\frac{n}{n-s},\infty),B(x_0,2^{k} \rho)}} 
\]
Taking the supremum over all such $\varphi$, see e.g. \cite[Proposition A.3.]{Blatt-Reiter-Schikorra-2016}, we obtain, possibly for a larger $k_0$,
\[
 \|w\|_{(\frac{n}{n-s},\infty),B(x_0,2^{-k_0}\rho)} \aleq \brac{\eps+\delta} \brac{\|w\|_{(\frac{n}{n-s},\infty),B(x_0,2^{k_0} \rho)} + \sum_{k=k_0}^\infty 2^{-k\sigma} \|w\|_{(\frac{n}{n-s},\infty),B(x_0,2^{k} \rho)}}. 
\]
Choosing $\eps$ and $\delta$ small enough this is a decay estimate that can be iterated on smaller and smaller balls, and gives the claim. See e.g. \cite[Lemma A.8]{Blatt-Reiter-Schikorra-2016}.

\section*{Acknowledgment}
A.S. is supported by the German Research Foundation (DFG) through grant no.~SCHI-1257-3-1. He receives funding from the Daimler and Benz foundation. A.S. was Heisenberg fellow.


\end{document}